\documentclass[preprint,12pt,3p]{elsarticle}

\usepackage{amsmath,amsthm,amssymb,amsfonts}
\usepackage{graphicx} 
\usepackage[normalem]{ulem}
\usepackage{color}
\usepackage{pifont}
\usepackage{mathrsfs}
\usepackage{mathtools}
\usepackage{float}
\newtheorem{theorem}{Theorem}
\newtheorem*{defn}{Definition}

\newtheorem{corr}{Corollary}
\newtheorem{lemma}{Lemma}
\newtheorem{prop}{Proposition}

\theoremstyle{remark}

\theoremstyle{plain}

\theoremstyle{plain}

\newcommand{\Z}{\mathbb{Z}}

\DeclarePairedDelimiter\ceil{\lceil}{\rceil}
\DeclarePairedDelimiter\floor{\lfloor}{\rfloor}

\journal{ }

\begin{document}

\begin{frontmatter}

\title{Self-Assembling DNA Complexes with a Wheel Graph Structure}

\author[label1]{Gabriel Lopez}
\address[label1]{California State University, San Bernardino, San Bernardino, CA}
\ead{gabriel.lopez.x15@gmail.com}

\author[label1]{Cory Johnson}
\ead{corrine.johnson@csusb.edu}

\begin{abstract}
The Watson-Crick complementary properties of DNA make DNA a useful tool for the self-assembly of various target complexes. Concepts from graph theory can be used to model the self-assembling  process in which the vertices of the graph represent $k$-armed branched junction molecules, called tiles. We seek to determine the minimum number of tile and cohesive-end types necessary to create the desired self-assembled complex. Although results are known for a few infinite classes of graphs, many classes of graphs remain unsolved. We present results for the wheel graph within the restrictions of three different settings.
\end{abstract}


\end{frontmatter}

%% MAIN CONTENT

\section{Introduction}
\subsection{Self Assembling DNA Complexes}

Advancements in nanotechnology have allowed biologists to construct various complexes from genetic material. This is accomplished by utilizing the complementary properties of DNA molecules so that the molecules bond in a specific way. These complexes, while proven to be incredibly useful, can also be costly to synthesize in a lab. One of our objectives is to make the creation of these structures more efficient. 

 There are several applications of DNA self-assembly including targeted drug delivery as a treatment for various types of cancer, gene therapy, the formation of fine screen filters, biomolecular computing, and biosensors (see \cite{2, 43, labean2007constructing, han, 80}). Our focus has been on complexes formed from flexible $k$-armed branched junction DNA molecules. While some general results are known, there are many open problems in finding the most optimal construction. See \cite{ellis2014minimal} for overview of problem and summary of main results. The notion of optimizing the construction of these complexes comes from the fact that creating a complex will require a certain number of bond-edge types, as well as a certain number of tile types. We use concepts from graph theory to model this design process. We then optimize the construction of a target graph by providing an algorithm that will construct the graph with as few tile and bond-edge types as possible under the constraints of one of the three following scenarios:  

\begin{itemize}
\item \emph{Scenario 1:} A pot of tile types may realize other graphs of smaller order than the target graph (that is, a graph on fewer vertices) or nonisomorphic graphs of the same size as the target graph. 

\item \emph{Scenario 2:} A pot of tile types may realize graphs which are of the same order as the target graph, but are nonisomorphic. Graphs of smaller order than the target graph cannot be realized by the pot. 

\item \emph{Scenario 3:} A pot of tile types may not realize graphs of smaller order nor any nonisomorphic graphs of the same order as the target graph.  
\end{itemize}

We introduce the following notation for the minimum number of tile types and bond-edge types needed to construct a complex. 

\begin{defn}
Let $G$ be a target graph. We denote the minimum number of bond-edge types needed to construct $G$ in Scenario $i$ by $B_i(G),$ and the minimum number of tile types needed to construct $G$ in Scenario $i$ by $T_i(G).$
\end{defn}

\subsection{Concepts from Graph Theory}

A DNA complex can    be modeled using a graph, a mathematical object consisting of a nonempty set of nodes, called vertices, and a set of edges  which connect pairs of distinct vertices.

Take as example the DNA complex in Figure \ref{DNAcube} which takes the shape of a cube. This structure can be modeled by the graph in Figure \ref{cubepic}.

 \begin{figure}[h]
   \includegraphics[width=4cm]{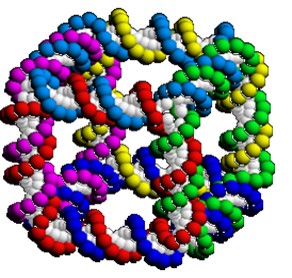}
    \centering
   \caption{A DNA complex.} \label{DNAcube}
    \end{figure}

 \begin{figure}[h!]
   \includegraphics[width=4cm]{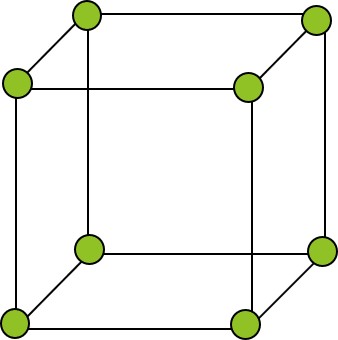}
    \centering
   \caption{The target graph associated with the complex in Figure \ref{DNAcube}.}\label{cubepic}  
    \end{figure}

The structure of a DNA complex becomes more easily studied as a graph rather than a complex made from genetic material. The $k$-armed DNA molecules that comprise the complex are modeled by objects called \emph{tiles}, which consist of a single vertex incident to $k$ half-edges. Each half-edge represents one arm of the molecule, with some paired sequence of DNA bases at the end. Since the end of each branch is unpaired, then in a lab environment each branch of the molecule seeks its complementary bond type to complete the sequence. We refer to a collection of tiles used to construct a graph as a \emph{pot} of tiles that \emph{realizes} the target graph. See an example of a branched molecule along with its corresponding tile in Figure \ref{Karm}. 
\begin{figure}[H]
	\centering
\includegraphics[trim=0 0 0 .5cm, clip, scale=.6]{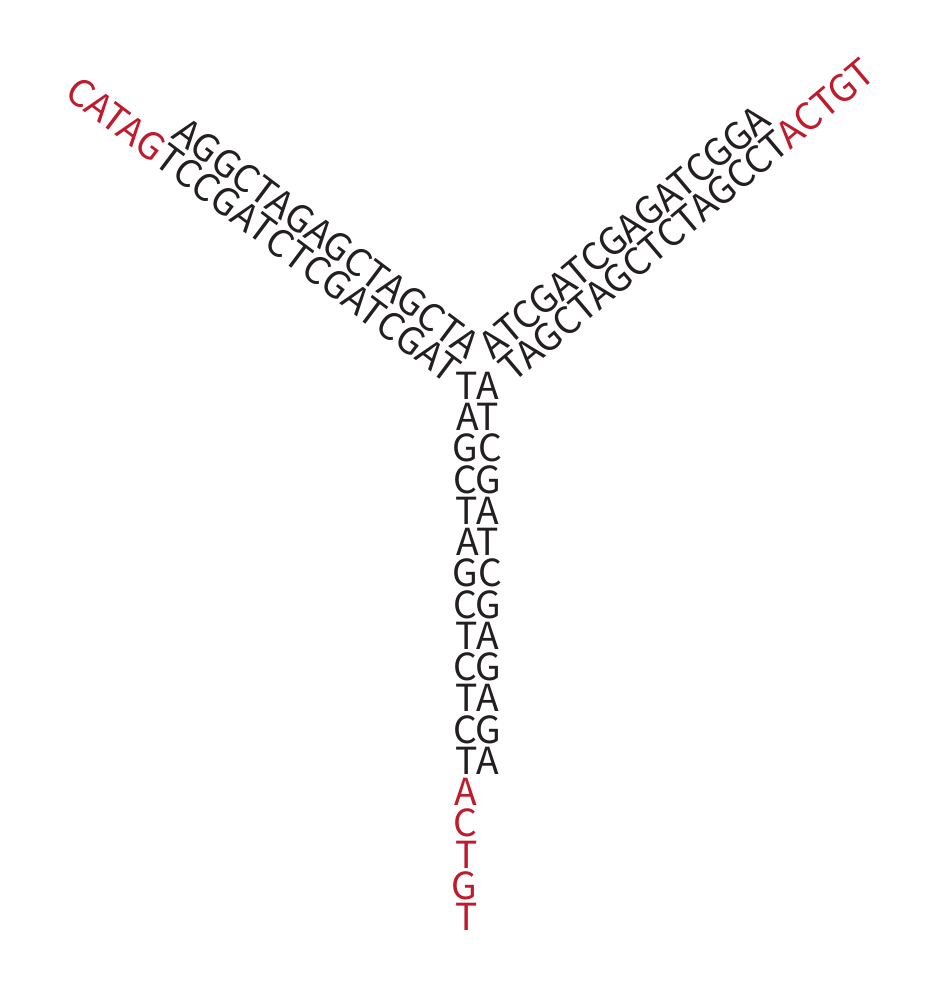} \hspace{.5in} \includegraphics[trim=0 0 0 8cm, clip, scale=.22]{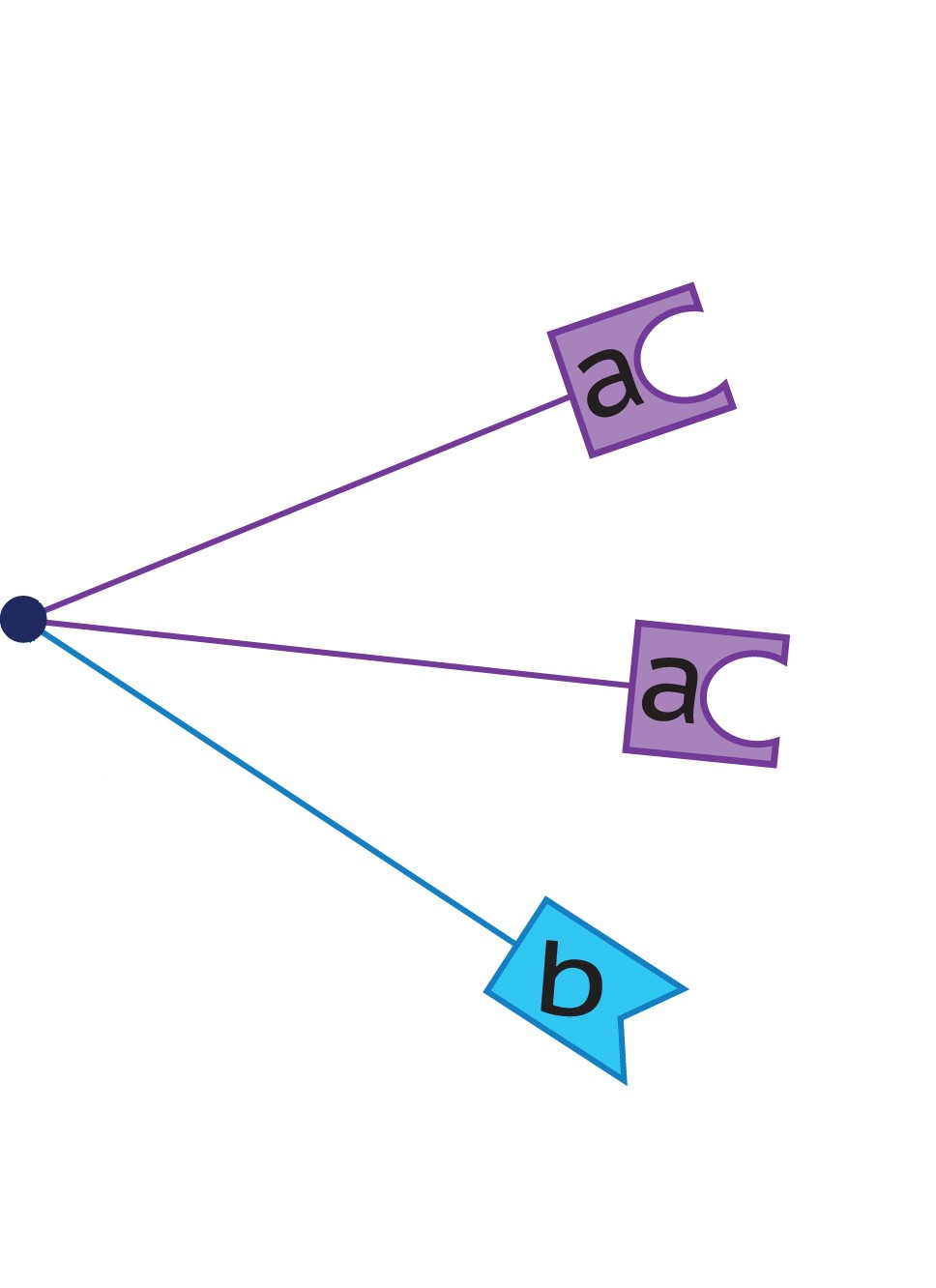}
	\caption{A 3-armed branched junction molecule (left) with example tile representation (right)}\label{Karm}
\end{figure}

Half edges are determined by their respective \emph{bond-edge types}, which we denote with some letter, with hatted and unhatted labels to denote pairs of complementary bond-edge types. For example, a half edge labeled $a$ must bond with a half edge labeled $\hat{a}$. In certain situations, it is useful to think of the half edges of each tile bonding together to form  directed edges. For the purpose of this paper, the edge will be directed from the vertex with the unhatted half edge to the vertex with the hatted half edge.  

Constructed using tiles rather than a plain graph, the cube complex in Figure \ref{DNAcube} could be represented by the graph on the left in Figure \ref{tilecube}. As a directed graph, the same complex may be constructed as seen on the right side of Figure \ref{tilecube}.

 \begin{figure}[h]
   \includegraphics[width=5cm]{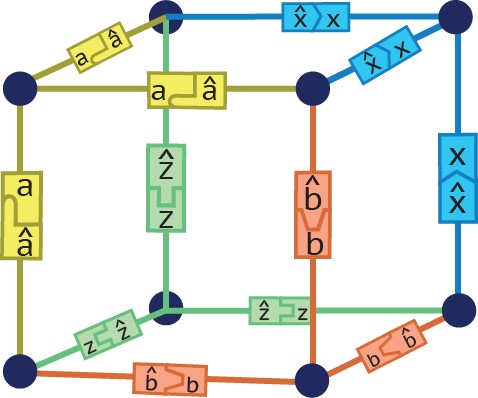} \hspace{.5in}
    \includegraphics[trim=0 1.6cm 0 0, clip, width=6cm]{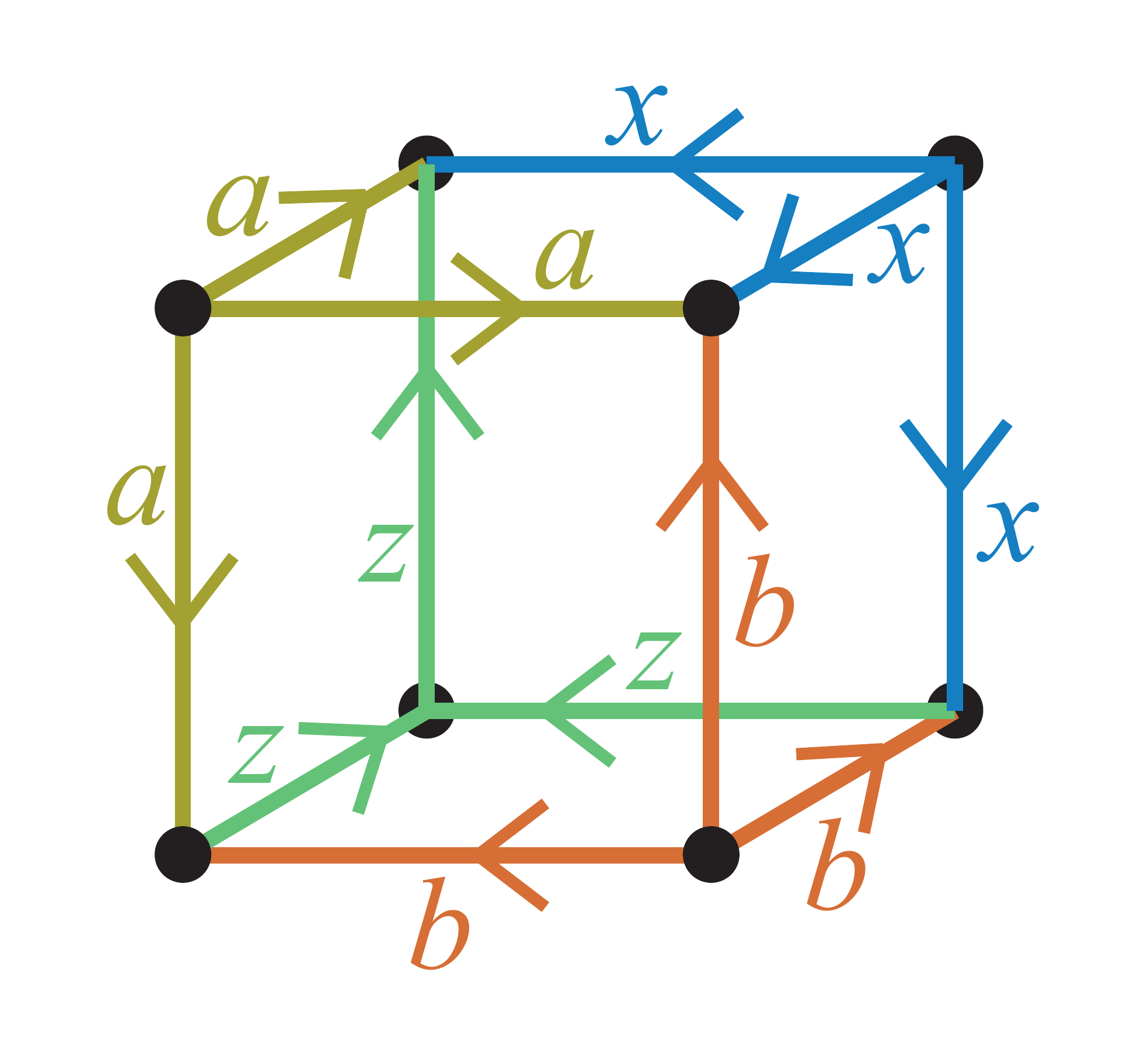}
    \centering
   \caption{A cube structure built with tiles (left), with equivalent digraph representation (right).}  \label{tilecube}
    \end{figure}

The process of modeling a self-assembling DNA complex with a (directed) graph is at the core of the work presented here. We seek to find a labeling of a target graph with a minimum alphabet of bond-edge types and a minimum set of distinct tiles. 

Shifting our focus from complexes to graphs, it is useful to have a few concepts from graph theory in mind.  We use the following graph theory concepts throughout this paper.

\begin{defn}
Let $G$ be a graph. We denote the number of different vertex degrees in $G$ by $av(G)$, the number of distinct even vertex degrees in $G$ by $ev(G)$, and the number of distinct odd vertex degrees by $ov(G).$
\end{defn}

\begin{defn}
A cycle in a graph $G$ that contains every vertex of $G$ is called a Hamilton cycle. Any graph which possesses a Hamilton cycle is called Hamiltonian.
\end{defn}

\subsection{The Construction Matrix}
We are able to use tools from linear algebra to aid us in verifying that a given pot of tiles has the desired properties. Specifically, there is a linear system that governs the ratio of tile and bond-edge types needed when constructing a target graph. This is due to the fact that for every half edge on every tile, there must be a complementary half edge of the same bond-edge type somewhere else in the pot. 

The following results are directly from \cite{ellis2014minimal} and are written here for later reference:

\begin{defn}
Let $P = \{t_1, \ldots, t_p\}$ be a pot of $p$ distinct tiles. We define $A_{i,j}$ to be the number of cohesive ends of type $a_i$ on tile $t_j$, and $\hat{A}_{i,j}$ to be the number of cohesive ends of type $\hat{a}_i$ on tile $t_j$. Let  $z_{i,j}$ be the net number of cohesive ends of type $a_i$ on tile type $t_j$, i.e. $z_{i,j}=A_{i,j}-\hat{A}_{i,j}$. Define $r_i$ to be the proportion of tile type $t_i$ to be used in the assembly process. \end{defn}

\begin{prop}\label{prop2}
Let $P=\{t_1,t_2,...t_p\}$ be a pot. Then the number of hatted cohesive ends of each bond type must equal the number of unhatted cohesive ends of the same type that appear in the construction of $G$. That is, for all $i$, $\sum_{j} r_jz_{i,j} = 0$.
\end{prop}

The following system of equations captures the requirements outlined in  Proposition \ref{prop2}.

\begin{align*}
    z_{1,1}r_1+z_{1,2}r_2+\dots+z_{1,p}r_p&=0 \\
    \vdots \\
    z_{m,1}r_1+z_{m,2}r_2+\dots+z_{m,p}r_p &=0 \\
    r_1+r_2+ \dots + r_p &= 1
\end{align*} 
The \emph{construction matrix} of $P$, denoted $M(P)$, is the corresponding augmented matrix:
\begin{equation}
    \begin{bmatrix}
     z_{1,1} & z_{1,2} & \dots & 0 \\
     \vdots & \vdots & \ddots & \vdots \\
     z_{m,1} & z_{m,2} & \dots & 0 \\
     1 & 1 & \dots & 1
    \end{bmatrix}.
\end{equation}

\begin{defn}
Given a pot $P$, we define $C(P)$ to be the set of graphs that can be constructed from $P$. 
The set of graphs of minimum order that may be constructed from $P$ is denoted $C_{min}(P)$. 
We write $m_P$ for the order of a smallest graph that may be constructed from $P$.
\end{defn}

\begin{prop}\label{prop3}
Let $P=\{t_1,t_2, \dots, t_p\}.$ Then:
\begin{enumerate}
    \item If a graph $G$ of order $n$ may be constructed from $P$, using $R_j$ tiles of type $t_j$, then $(1/n)\langle R_1,R_2,\dots,R_p\rangle$ is a solution of the construction matrix $M(P).$ 
    \item If $\langle r_1, \dots, r_p \rangle$ is a solution of the construction matrix $M(P)$, and there is a positive integer $n$ such that $nr_j \in \Z_{\geq0}$ for all $j$, then there is a graph of order $n$ that may be constructed using $nr_j$ tiles of type $t_j$.
    \item $m_p = \textrm{min}\{lcm\{b_j|r_j \neq 0 \textrm{ and } r_j = a_j/b_j\}, \textrm{where } \langle r_1, \dots, r_p \rangle \textrm{ is a solution to } M(P)\}$ where the minimum is taken over all solutions to $M(P)$ such that $r_j \geq 0$ and $a_j/b_j$ is in reduced form for all $j$.
\end{enumerate}
\end{prop}

This paper describes results for target complexes that can be modeled with a wheel graph.

\subsection{Wheel Graphs}
A \emph{wheel graph} on $n$ vertices, denoted $W_n$, is a graph consisting of a cycle on $n-1$ vertices and an additional vertex (often called the \emph{hub}) which is adjacent to every vertex on the cycle. 
    \begin{figure}[h!]
   \includegraphics[width=4cm]{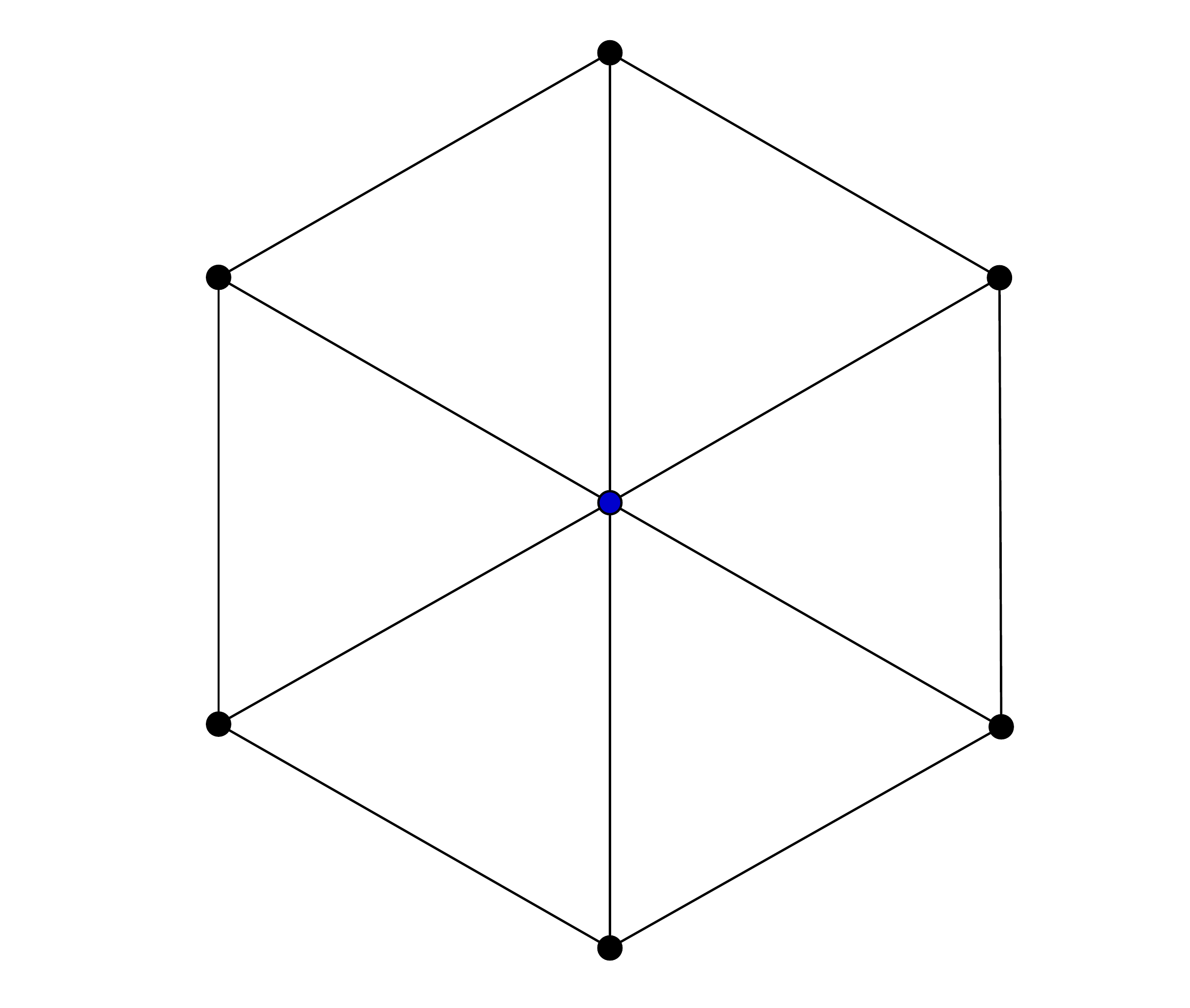}
    \centering
   \caption{The wheel graph $W_7$.}  
    \end{figure}
For the purpose of this paper, we refer to the $n-1$ cycle subgraph as the \emph{outer cycle}. The edges that connect the hub to the vertices on the outer cycle are called \emph{spokes}.  

A property of wheel graphs that we use later on in Section \ref{scen3results} is that they are Hamiltonian. One possible Hamilton cycle is shown in Figure \ref{hamiltonpic}.
\begin{figure}
    \centering
    \includegraphics[width=7cm]{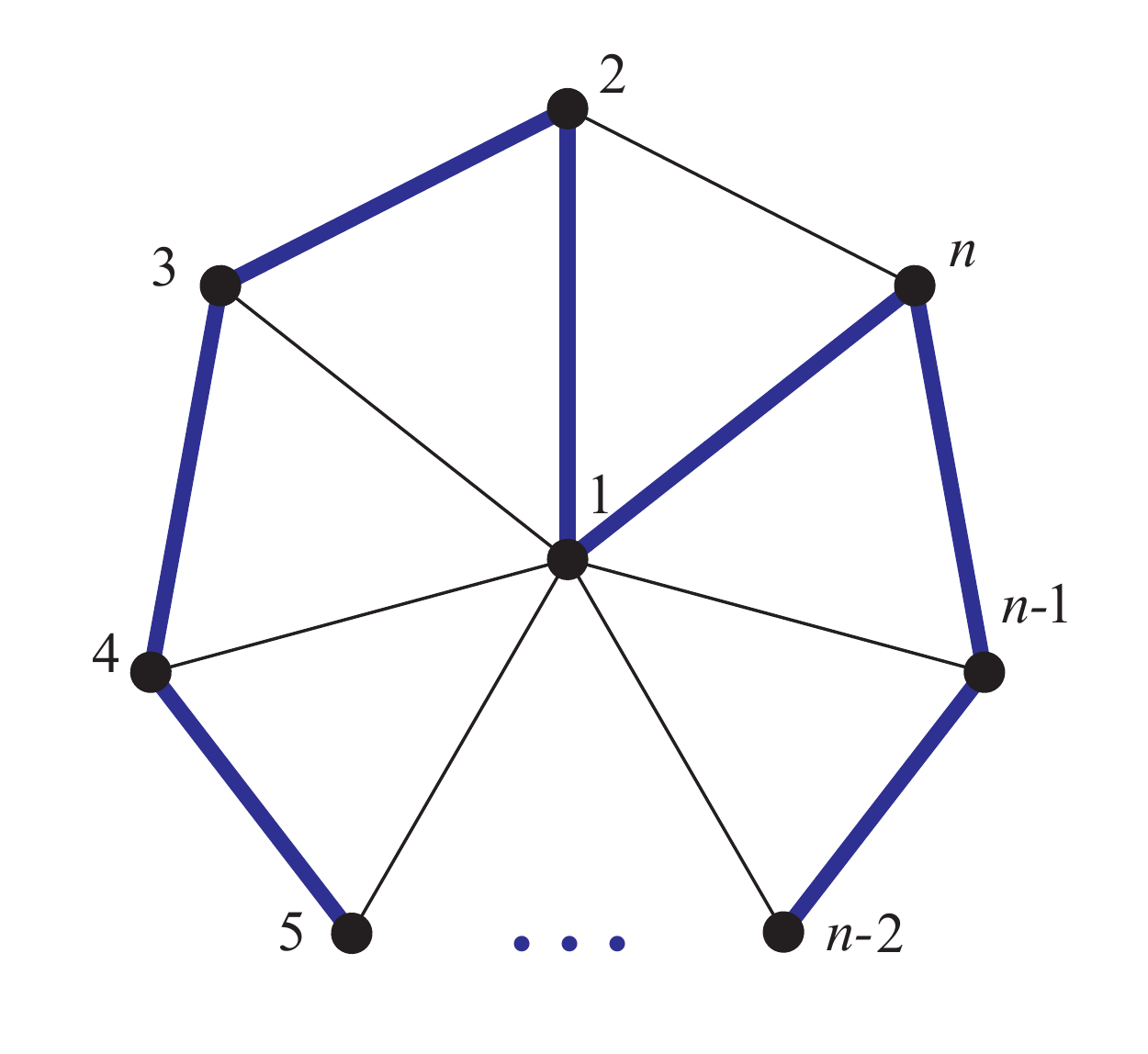}
    \caption{The Hamilton cycle in a wheel graph visits all $n$ vertices of $W_n$.}
    \label{hamiltonpic}
\end{figure}  
In the case of a wheel graph on 4 vertices, we note that $W_4$ is isomorphic to the complete graph on 4 vertices, $K_4.$ All results presented in this paper for the case where $n=4$ are consistent with the results for $K_4$ in \cite{ellis2014minimal}. Unless otherwise stated, this paper assumes any wheel graph  has order greater than 4; that is, we assume $n > 4$.  

\section{Methods}
 Previous authors have provided a number of useful theorems that we use to prove new results on the construction of wheel graphs. The following are results from \cite{ellis2014minimal}:

 \begin{theorem}\label{thm1}
 $av(G)\leq T_1(G) \leq ev(G)+2ov(G)$
 \end{theorem}
 
 Theorem \ref{thm1} tells us that in Scenario 1, the number of distinct tile types needed to construct a target graph is bounded below by the number of distinct vertex degrees that are present in the graph, and bounded above by the sum of the number of distinct even degrees and twice the number of distinct odd degrees present in the graph. The lower bound holds for $T_2(G)$ and $T_3(G)$ as well. \newline

 We also take advantage of the following preposition regarding the hierarchy of tile type and bond-edge type minima between scenarios.
 \begin{prop} \label{hierarchy}
 $B_1(G)\leq B_2(G)\leq B_3(G)$ and  $T_1(G)\leq T_2(G)\leq T_3(G)$.
 \end{prop}
 
 The next theorem gives the minimum number of tile and bond-edge types needed to construct the cycle graph on $n$ vertices, $C_n$, in Scenario 3. This result is useful since the cycle graph is a subgraph of the wheel graph.

 \begin{prop}\label{cycle3}
 $B_3(C_{n})=\ceil*{\frac{n}{2}}$ and $T_3(C_{n})=\ceil*{\frac{n}{2}}+1$.
 \end{prop}

 The following result is another lemma that we use in the proof of one of our theorems.

 \begin{lemma} \label{lemma3}
  In Scenario 3, if $P$ is a pot such that $\{G\} = C_{min}(P)$, and $G$ has no loops, then no tile type $t \in P$ may be used for two adjacent vertices in $G$.   
 \end{lemma}

 \section{Results}
 
 \subsection{Wheel Graphs in Scenario 1 and Scenario 2}
We have developed an algorithm to construct wheel graphs in Scenarios 1 and 2 in the most efficient way possible. We use a pot of two tile types and a single bond-edge type to construct a wheel graph of any order. 

The pot we propose is given below:
\begin{equation}\label{PotforScen1and2}
    P= \{t_1, t_2\},
\end{equation}
where 
$$t_1=\{a,\hat{a}^2\}, \textrm{and } t_2=\{a^{n-1}\}.$$
%$$t_2=\{a^{n-1}\},$$
This leads us to our first theorem.

\begin{theorem}
Let $W_n$ be a wheel graph on $n$ vertices. Then $T_1(W_n)=2$ and $B_1(W_n)=1$.
\end{theorem}

\begin{figure}[H]
   \includegraphics[width=9cm]{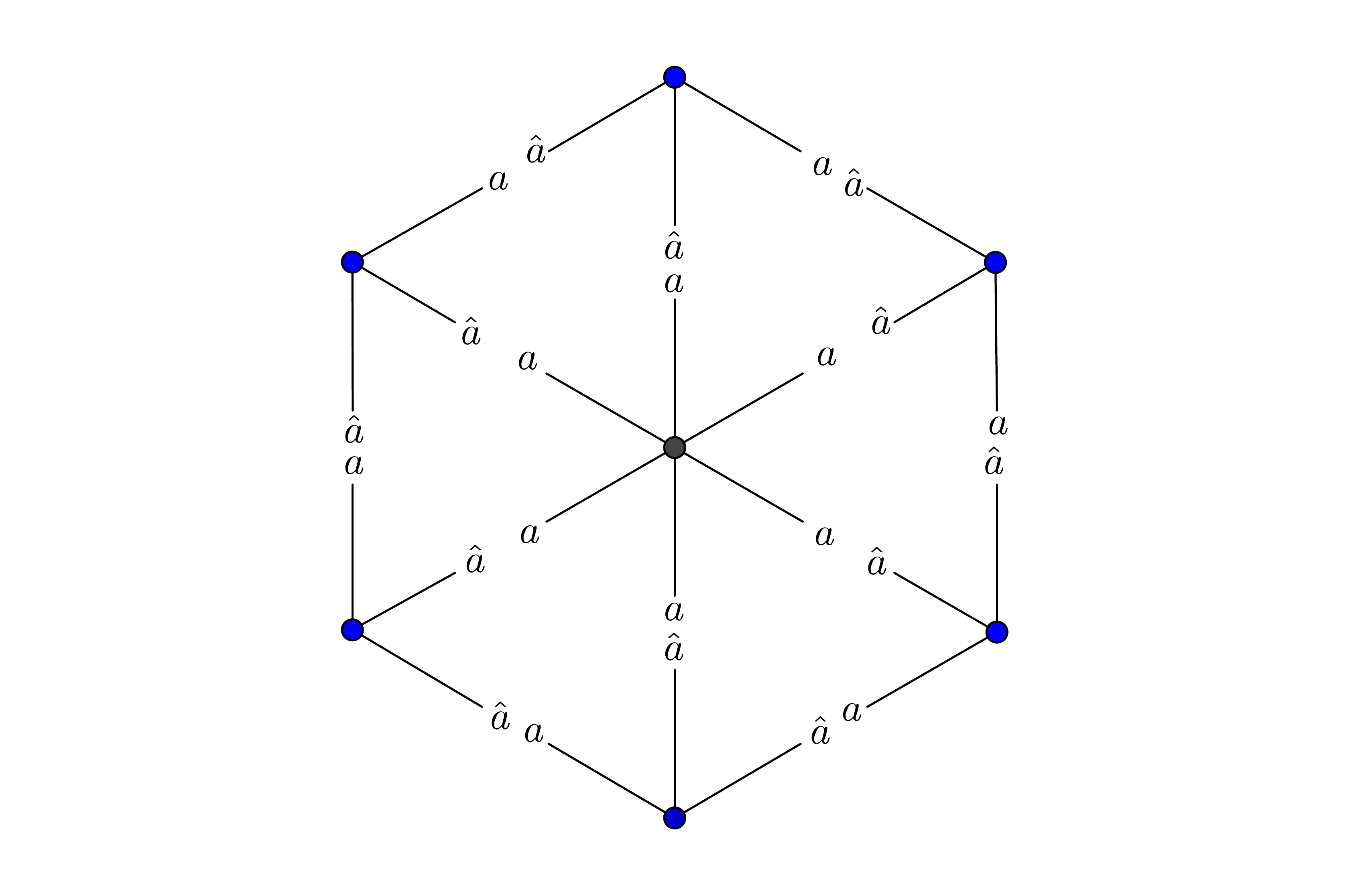}
   \centering
   \caption{$W_7$ constructed using the pot of tiles in Equation (\ref{PotforScen1and2}).}
\end{figure}

 \begin{proof}
 Let $W_n$ denote a wheel graph on $n$ vertices. From \cite{ellis2014minimal}, we know that $B_1(W_n)=1$. We observe that $av(W_n)=2$, so by Theorem \ref{thm1}, we know that $2\leq T_1(W_n)$. This lower bound is achieved by the pot proposed above. We may now conclude that $T_1(W_n)=2$.
 \end{proof}
 
\begin{corr}\label{scen2}
Let $W_n$ be a wheel graph on $n$ vertices. Then $T_2(W_n)=2$ and $B_2(W_n)=1$.
\end{corr}

\begin{proof}
By Theorem \ref{thm1} and Proposition \ref{hierarchy}, $2\leq T_1(W_n) \leq T_2(W_n)$. This lower bound is achieved by the pot proposed above.  We will now use the construction matrix to verify that no smaller graphs may be constructed from the same pot. The pot of tiles in (\ref{PotforScen1and2}) gives us the following construction matrix:
$$
\begin{bmatrix}
    -1 & n-1 & 0  \\
    1 & 1 & 1  \\
\end{bmatrix},
$$
which yields the unique solution $\langle \frac{n-1}{n},\frac{1}{n} \rangle$, thus by Proposition \ref{prop3} a graph on $n$ vertices is the smallest graph realized by this pot. Therefore, $T_2(W_n)=2$ and $B_2(W_n)=1$.
\end{proof}

It is worth noting that while other graphs on $n$ vertices nonisomorphic to the wheel graph are able to be constructed from this same pot of tiles, the conditions of Scenarios 1 and 2 do not prohibit this.

\begin{figure}[h!]
    \centering
    \includegraphics[width=7cm]{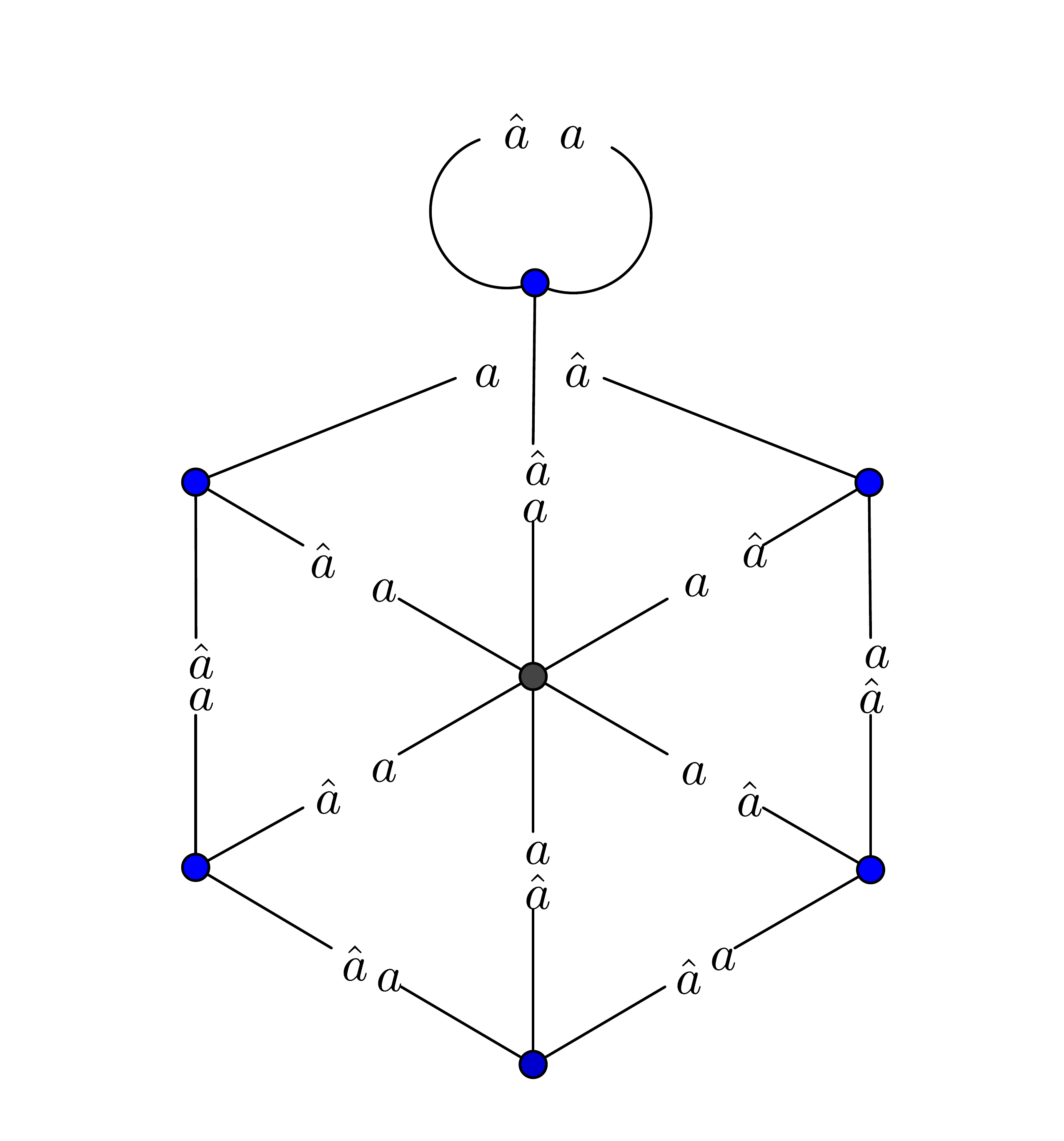}
    \caption{Nonisomorphic graph realized by the pot in Scenario 2}
    \label{nonisograph}
\end{figure}

 \subsection{Wheel Graphs in Scenario 3} \label{scen3results}

Wheel graphs in Scenario 3 require a different construction algorithm  compared to the construction in Scenarios 1 and 2 due to the stricter condition that no nonisomorphic graphs of the same order may be realized by a pot of tiles. To begin, we introduce two lemmas that aid us in optimizing the construction of $W_n$. 
 
 %%%%%%%%%%%%%%%%%%%%%%%%%%%%%%%%%%%%%%%%%%%%%%%%%%%%%%%%%%%%%%%%%%%%%%%%%%%%
%LEMMA J
%%%%%%%%%%%%%%%%%%%%%%%%%%%%%%%%%%%%%%%%%%%%%%%%%%%%%%%%%%%%%%%%%%%%%%%%%%%%
 \begin{lemma}\label{lemmaj}
 If $P$ is a pot such that $\{W_n\} = C_{min}(P)$, then no bond-edge type used in the construction of $W_n$ may be used both on the outer cycle and a nonincident spoke. 
 \end{lemma}

 \begin{proof}
 We proceed by contradiction. Suppose there exists a pot $P$ such that $\{W_n\} = C_{min}(P)$ and some bond-edge type, say $a$, appears on an edge on the outer cycle as well as a nonincident spoke edge. See the graphs on the left side of Figure \ref{lemmajcontra}. 
 
 \begin{figure}[h]
     \centering
     \includegraphics[width=10cm]{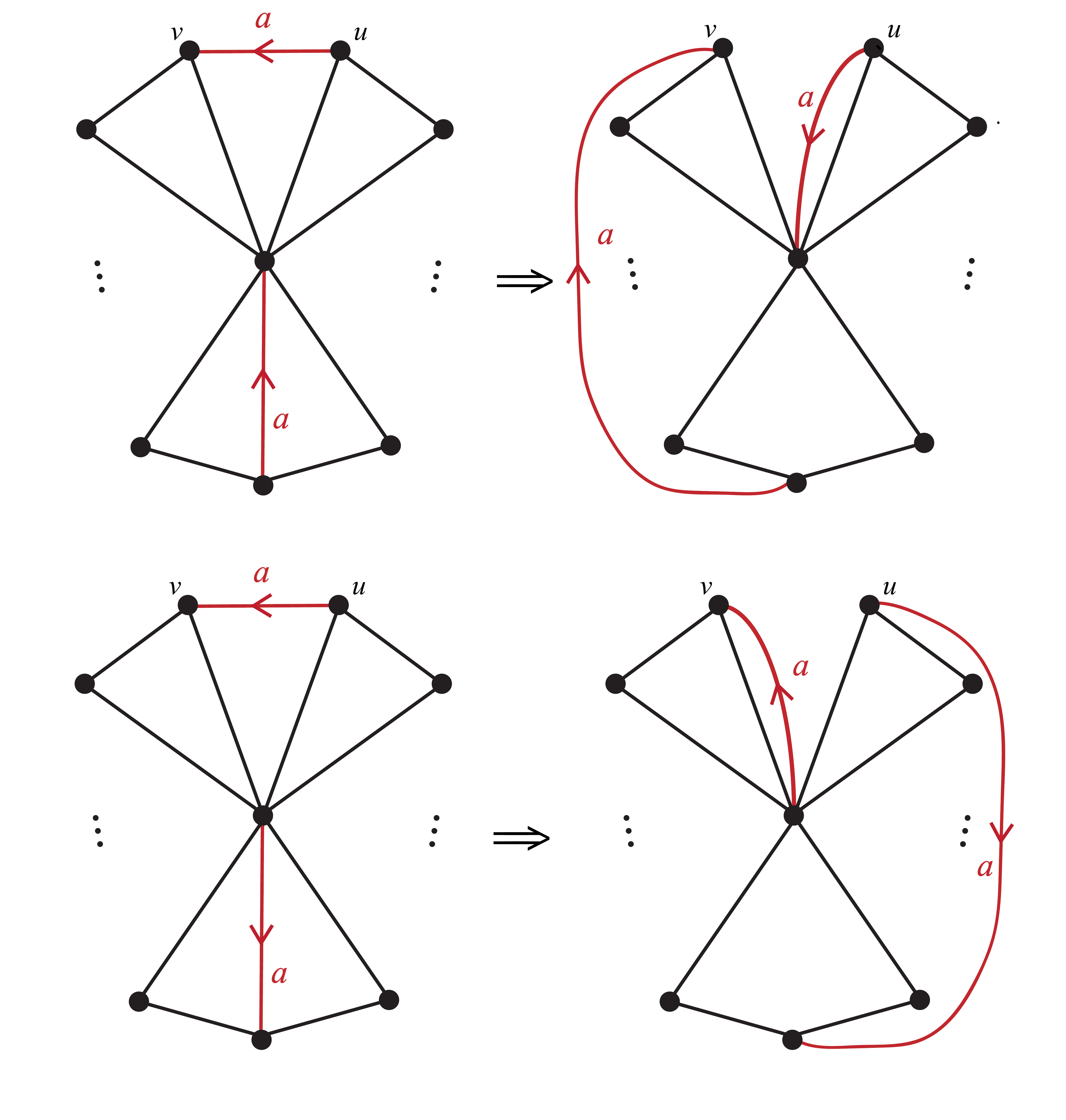}
     \caption{Graph with bond-edge type on the outer cycle and a nonincident spoke edge}
     \label{lemmajcontra}
 \end{figure}
 
 Without loss of generality, we may assume the edge on the outer cycle is oriented counterclockwise,  call this edge $(u,v)$. The spoke edge with the same bond-edge type may be oriented either towards the hub or towards the outer  cycle, but in either case, it is possible that the half edges on the spoke may reattach with the half edges on the outer cycle. This will produce a multiple edge with the hub and  vertex $u$ or $v$, as seen on the right side in Figure \ref{lemmajcontra}. The resulting graph is nonisomorphic to $W_n$, thus contradicting the assumption that $\{W_n\} = C_{min}(P)$.
 \end{proof}
 
 %%%%%%%%%%%%%%%%%%%%%%%%%%%%%%%%%%%%%%%%
 %END LEMMA J
 %%%%%%%%%%%%%%%%%%%%%%%%%%%%%%%%%%%%%%%
 
  %%%%%%%%%%%%%%%%%%%%%%%%%%%% CYCLE LEMMA
   \begin{lemma}\label{cyclelemma}
If $P$ is a pot such that $\{W_n\} = C_{min}(P)$, then a bond-edge type can be used  at most twice on the outer cycle in the construction of $W_n$.
 \end{lemma}
 
 %%%%%%%%%%%%%%%%%%%%%%%%%%%%%%%%%%%%%%%%%%%%%%%%%%%%%%%%%%%%%%%%%%%%%%%%%
 %PROOF OF CYCLE LEMMA
 %%%%%%%%%%%%%%%%%%%%%%%%%%%%%%%%%%%%%%%%%%%%%%%%%%%%%%%%%%%%%%%%%%%%%%%%%

 \begin{proof}
By way of contradiction, suppose $\{W_n\} = C_{min}(P)$ and suppose in the construction of $W_n$ more than two edges along the outer cycle are labeled with the same bond-edge type. By the Pigeonhole Principle, at least two of these edges must be directed in the same orientation. Without loss of generality, suppose these two edges are oriented clockwise. See Figure \ref{cyclelemmaproof} for reference.  

\begin{figure}[h]
    \centering
    \includegraphics[width=7cm]{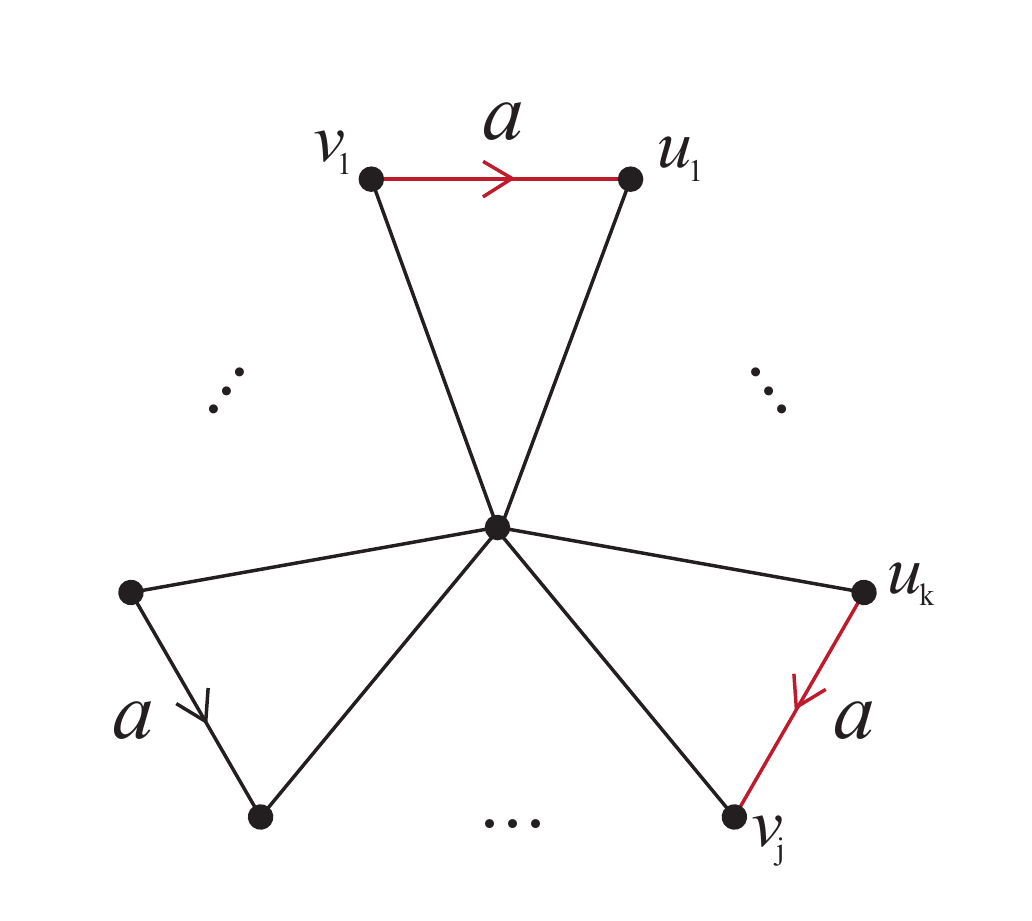}
    \caption{Graph with bond-edge type appearing at least three times on the outer cycle}
    \label{cyclelemmaproof}
\end{figure}

This arrangement allows for the edges to detach and rejoin as in Figure \ref{cyclecontra}  to form the complex $H$. We will now show that this will necessarily yield a nonisomorphic graph. 
 
   \begin{figure}[h]
     \includegraphics[ width=8cm]{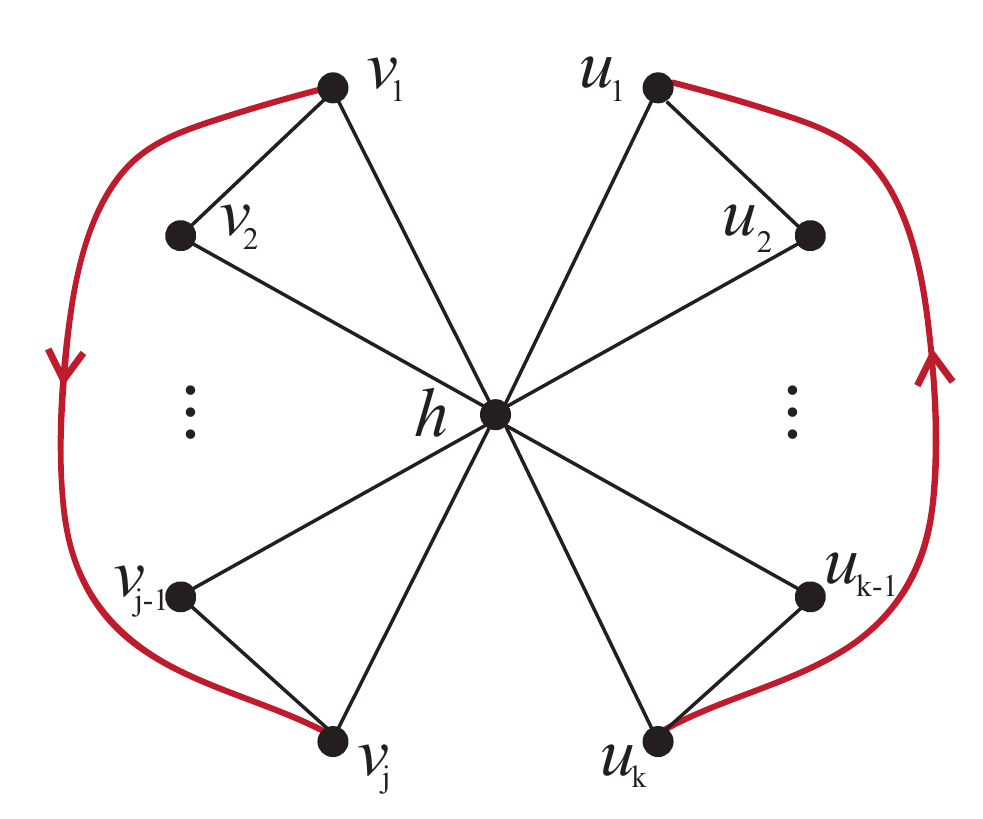}
     \centering
     \caption{Resulting graph $H$}
     \label{cyclecontra}
 \end{figure}
 
 Recall, $W_n$ has a Hamilton cycle, as seen in Figure \ref{hamiltonpic}. If graph $H$ is isomorphic to the original wheel graph, then it must have a Hamilton cycle.  
 
 However, refer to  Figure \ref{cyclecontra} to see that the set of vertices excluding the hub, $h$, may be partitioned into two sets so that no two vertices in different sets are adjacent. Suppose without loss of generality that we start a walk at vertex $h$. 
 From there, a walk of distinct vertices will include each vertex, $u_i$ for $1 \leq i \leq k$. The only way to include vertex $v_1$ in a cycle is for the walk to pass through vertex $h$. Thus, no Hamilton cycle exists in the graph $H$. Therefore, the pot $P$ realizes a graph nonisomorphic to $W_n$, contradicting the assumption that $\{W_n\} = C_{min}(P)$. 
 \end{proof}

 With these new lemmas, we have developed an algorithm to construct a wheel graph that builds upon the  construction of the cycle graph in Scenario 3 (see Proposition \ref{cycle3}). In particular, to construct the wheel graph $W_n$, we begin with the construction of the cycle graph $C_{n-1}$, pictured in Figure \ref{scen3cycle}.
     \begin{figure}[H]
    \includegraphics[trim=0 1cm 0 1cm, clip, width=6cm]{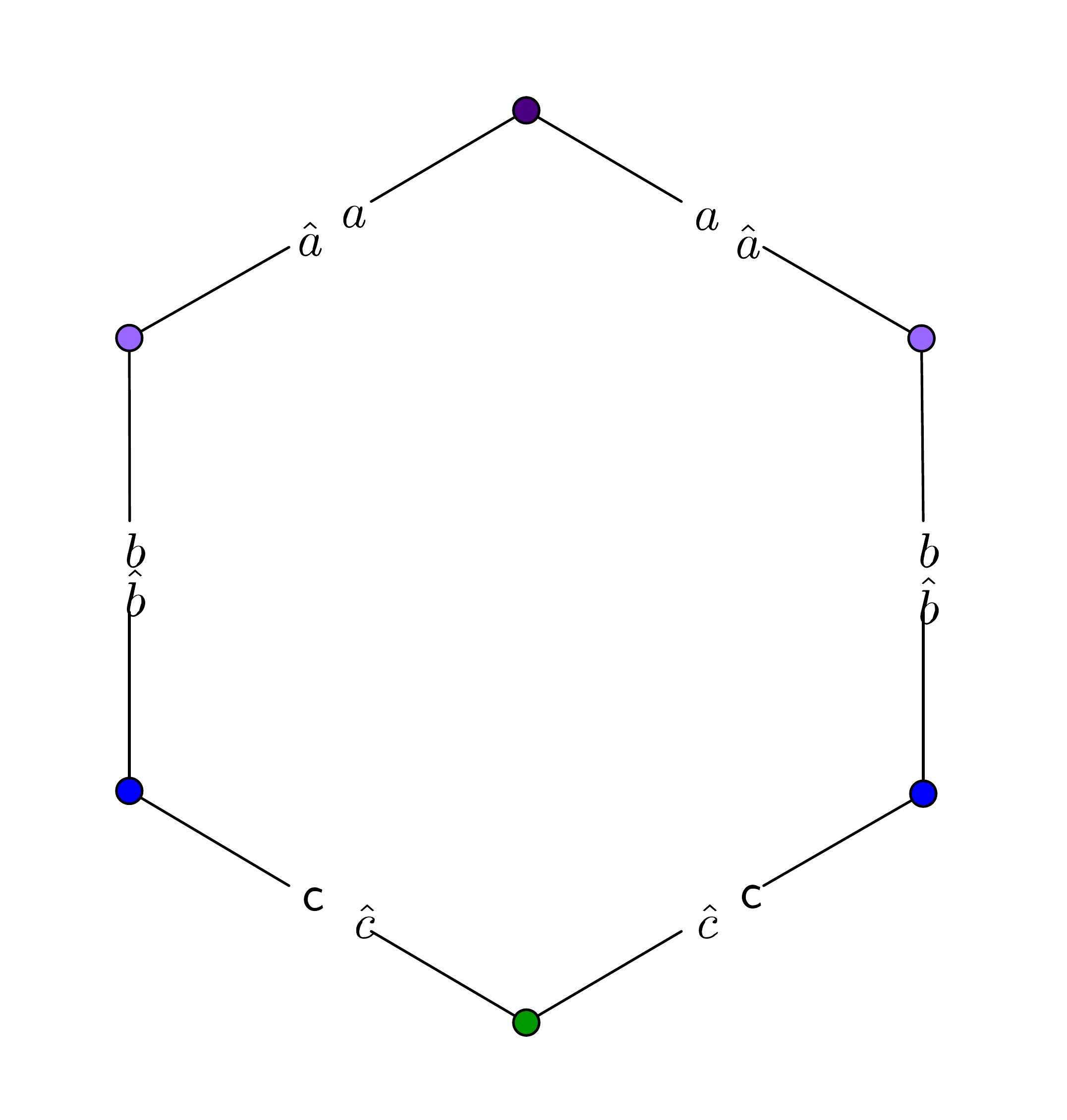} 
    \centering
        \caption{The optimized cycle graph $C_6$.}\label{scen3cycle}
    \end{figure}

From the optimized construction of the cycle graph, we can append the hub vertex and all necessary edges to the vertices on the outer cycle with one new bond-edge type.
  \begin{figure}[h!]
   \includegraphics[trim=0 0 0 1cm, clip, width=8.5cm]{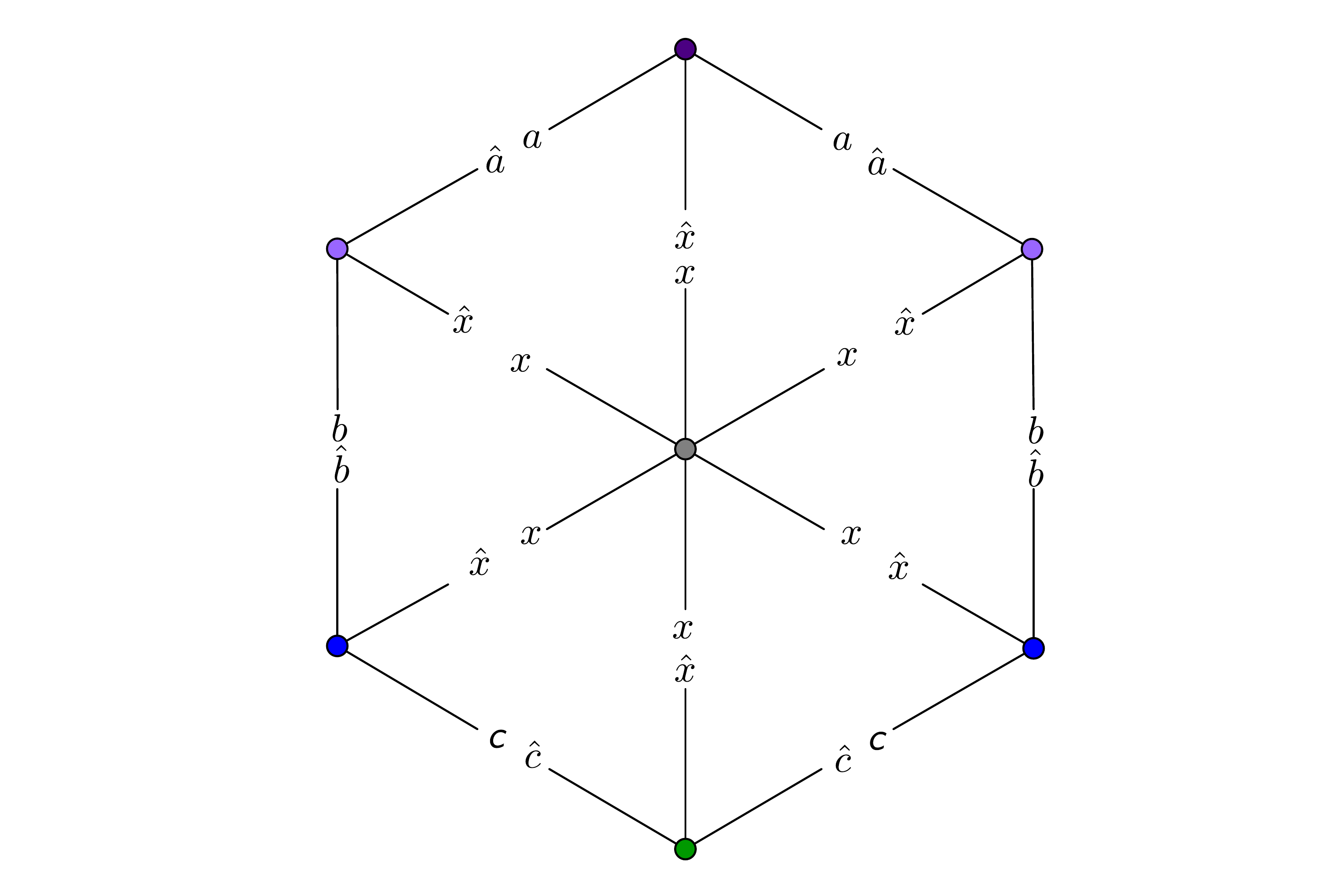} \
   \centering
   \caption{The proposed optimization of the wheel graph.}\label{scen3wheel}
    \end{figure}

 To prove that the construction described above uses the minimum number of tile types and bond-edge types needed to construct the wheel graph on $n$ vertices, it suffices to show that the proposed labeling constructs the wheel graph as efficiently as possible, i.e., that it is not possible to construct $W_n$ with fewer bond-edge types or tile types in Scenario 3. 
 %%%%%%%%%%%%%%%%%%%%%%%%%%%%%%%%%%%%%%%%%%%%%%%%%%%%%%%%%%%%%%%%%%%%%%%
 %PROOF OF THEOREM 3
 %%%%%%%%%%%%%%%%%%%%%%%%%%%%%%%%%%%%%%%%%%%%%%%%%%%%%%%%%%%%%%%%%%%%%%%
 \begin{theorem}\label{scen3}
$B_3(W_n)=\floor*{\frac{n}{2}}+1$. % and $T_3(W_n) = \floor*{\frac{n}{2}}+2$.
\end{theorem}

\begin{proof}
In the wheel graph, $W_n$, there are $n-1$ edges on the outer cycle. By Lemma \ref{cyclelemma}, no bond-edge type on the outer cycle may be used more than twice. Thus, $B_3(W_n) \geq \ceil*{\frac{n-1}{2}} = \floor*{\frac{n}{2}}.$  

If any bond-edge type from the outer cycle is used on more than one spoke edge, then that bond-edge type will appear on both the outer cycle and some nonincident spoke. By Lemma \ref{lemmaj}, this is not permitted. Since there are $n-1$ spoke edges, then it is necessary to introduce at least one new bond-edge type. Thus, $B_3(W_n) \geq \floor{\frac{n}{2}}+1.$  

The following pots of tiles realize $W_n$ using exactly $\floor{\frac{n}{2}}+1$ bond-edge types. Note that $P_{even}$ and $P_{odd}$ are the pots when $n$ is even or odd, respectively.

\begin{align}\label{Scen3even}
P_{even}= \Big\{ &t_1=\{a_1^{n-1}\}, t_2=\{\hat{a}_1, a_2^2 \}, t_i=\{\hat{a}_1, \hat{a}_{i-1}, a_i \},  \text{for } i = 3, \ldots, \floor*{\frac{n}{2}}+1, \\
& t_{\floor{\frac{n}{2}}+2}=\{\hat{a}_1, \hat{a}_{\floor{\frac{n}{2}}}, \hat{a}_{\floor{\frac{n}{2}}+1}\} \Big\} \nonumber
\end{align}

\begin{align}\label{Scen3odd}
   P_{odd}= \Big\{&t_1=\{a_1^{n-1}\}, t_2=\{\hat{a}_1, a_2^2 \}, t_i=\{\hat{a}_1, \hat{a}_{i-1}, a_i \},  \text{for } i = 3, \ldots, \floor*{\frac{n}{2}}+1, \\
   &t_{\floor{\frac{n}{2}}+2}=\{\hat{a}_1, \hat{a}_{\floor{\frac{n}{2}}+1}^2\} \Big\} \nonumber
\end{align}

The pot of tiles where $n$ is even produces the following construction matrix:
\begin{equation} \label{evenmatrix}
M(P_{even})=
\begin{bmatrix}
 (n-1)  & -1 & -1 & -1 & -1 & \cdots & -1 & -1 & 0\\ 
 0& 2 & -1 & 0 & 0 & \cdots & 0 & 0 & 0\\ 
 0& 0 & 1 & -1 & 0 & \cdots & 0 & 0 & 0\\ 
 0& 0 & 0 & 1 & -1 & \cdots & 0 & 0 & 0\\ 
 0& 0 & 0 & 0 & 1 & \cdots & 0 & 0 & 0\\ 
 \vdots& \vdots & \vdots & \vdots & \vdots & \ddots & \vdots & \vdots & \vdots\\ 
 0& 0 & 0 & 0 & 0 & \cdots & -1 & -1 & 0\\ 
 0& 0 & 0 & 0 & 0 & \cdots & 1 & -1 & 0\\ 
1 & 1 & 1 & 1 & 1 & 1 & 1 & 1 & 1
\end{bmatrix}    
\end{equation}

This construction matrix has the unique solution $\langle \frac{1}{n}, \frac{1}{n}, \frac{2}{n},  \dots, \frac{2}{n}, \frac{1}{n}, \frac{1}{n}\rangle $.

The pot of tiles where $n$ is even produces the following construction matrix:
\begin{equation} \label{oddmatrix}
  M(P_{odd}) = \begin{bmatrix} 
 (n-1)  & -1 & -1 & -1 & -1 & \cdots & -1 & -1 & 0\\ 
 0& 2 & -1 & 0 & 0 & \cdots & 0 & 0 & 0\\ 
 0& 0 & 1 & -1 & 0 & \cdots & 0 & 0 & 0\\ 
 0& 0 & 0 & 1 & -1 & \cdots & 0 & 0 & 0\\ 
 0& 0 & 0 & 0 & 1 & \cdots & 0 & 0 & 0\\ 
 \vdots& \vdots & \vdots & \vdots & \vdots & \ddots & \vdots & \vdots & \vdots\\ 
 0& 0 & 0 & 0 & 0 & \cdots & -1 & 0 & 0\\ 
 0& 0 & 0 & 0 & 0 & \cdots & 1 & -2 & 0\\ 
1 & 1 & 1 & 1 & 1 & 1 & 1 & 1 & 1
\end{bmatrix}
\end{equation}
This construction matrix has the unique solution $\langle \frac{1}{n}, \frac{1}{n}, \frac{2}{n},  \dots, \frac{2}{n}, \frac{1}{n}\rangle $. By Proposition \ref{prop3}, the smallest graph that can be realized by this pot is on $n$ vertices.

We must now show that the proposed pots will not realize graphs nonisomorphic to $W_n$. In both proposed pots, the solutions from (\ref{evenmatrix}) and (\ref{oddmatrix}) give the ratio of tile types to be used when constructing a graph of order $n$. In both pots, only one copy of tile $t_1$ is used, and since every other tile has one bond-edge type of type $\hat{a_1}$, and $t_1$ is only tile containing a bond-edge of type $a_1$, then $t_1$ must bond with every other tile from $P$ in order to form a complete complex.  In both even and odd cases, tile type $t_2$ is used exactly once. The only tile type that is able to bond with the unattached edges of $t_2$ is $t_3$, so the two copies of $t_3$ must bond with $t_2$. Once the two copies of $t_3$ bond to $t_2$ and to $t_1$, then each of these two copies has a free cohesive end that can only be complemented by a cohesive end from $t_4$, so each copy of $t_4$ must bond with one of the copies of $t_3$. In general, for $3 \leq i \leq \floor{\frac{n}{2}}+1,$ tile $t_i$ can only bond with tiles $t_1$, $t_{i-1},$ and $t_{i+1}$. For cases in which $n$ is even, $t_{\floor{\frac{n}{2}}+2}$ must bond with tiles $t_{\floor{\frac{n}{2}}}$ and $t_{\floor{\frac{n}{2}}+1}$. For cases in which $n$ is odd, $t_{\floor{\frac{n}{2}}+2}$ must bond with two copies of tile $t_{\floor{\frac{n}{2}}+1}$. The resulting complex is isomorphic to $W_n$.

Thus, $B_3(W_n) = \floor*{\frac{n}{2}}+1$.
\end{proof}

The corresponding result for minimum number of tile types needed to construct $W_n$ follows from Theorem \ref{scen3}.

\begin{theorem}\label{scen3tiles}
$T_3(W_n)=\floor{\frac{n}{2}}+2$.
\end{theorem}

\begin{proof}
Since the only vertex of degree $n-1$ is the hub, this vertex must have its own tile type. 

The pots in (\ref{Scen3even}) and (\ref{Scen3odd}) suggest that the remaining vertices, all of which are of degree 3, can be labeled with $\floor{\frac{n}{2}} + 1$ tile types. 

If it were possible to construct the outer cycle of $W_n$ using exactly $\floor{\frac{n}{2}}$ tile types, then every tile on the outer cycle would appear twice if $n$ is odd, and if $n$ is even, exactly one tile type on the outer cycle would appear once.
Suppose for sake of contradiction that $\floor{\frac{n}{2}}$ tile types may be used on the outer cycle. Label some pair of tiles on the outer cycle $t_1$, where $t_1=\{a,b,c\}$ (where $a, b, c$ are not necessarily distinct). By Lemma \ref{lemma3}, these two copies must have a distance between them of at least two.

\begin{figure}[H]
	\centering
\includegraphics[scale=.4]{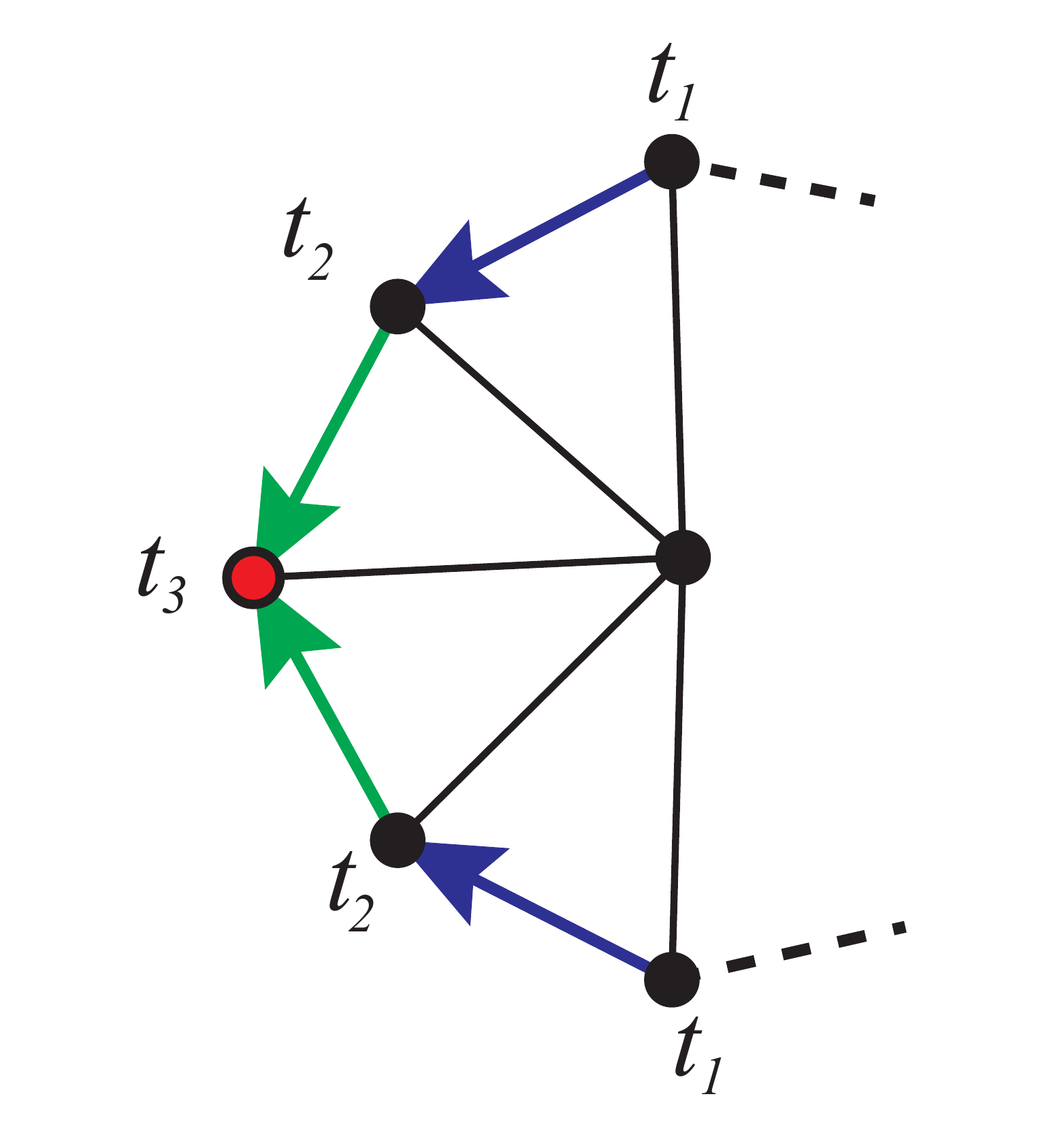} \hspace{.5in} \includegraphics[scale=.4]{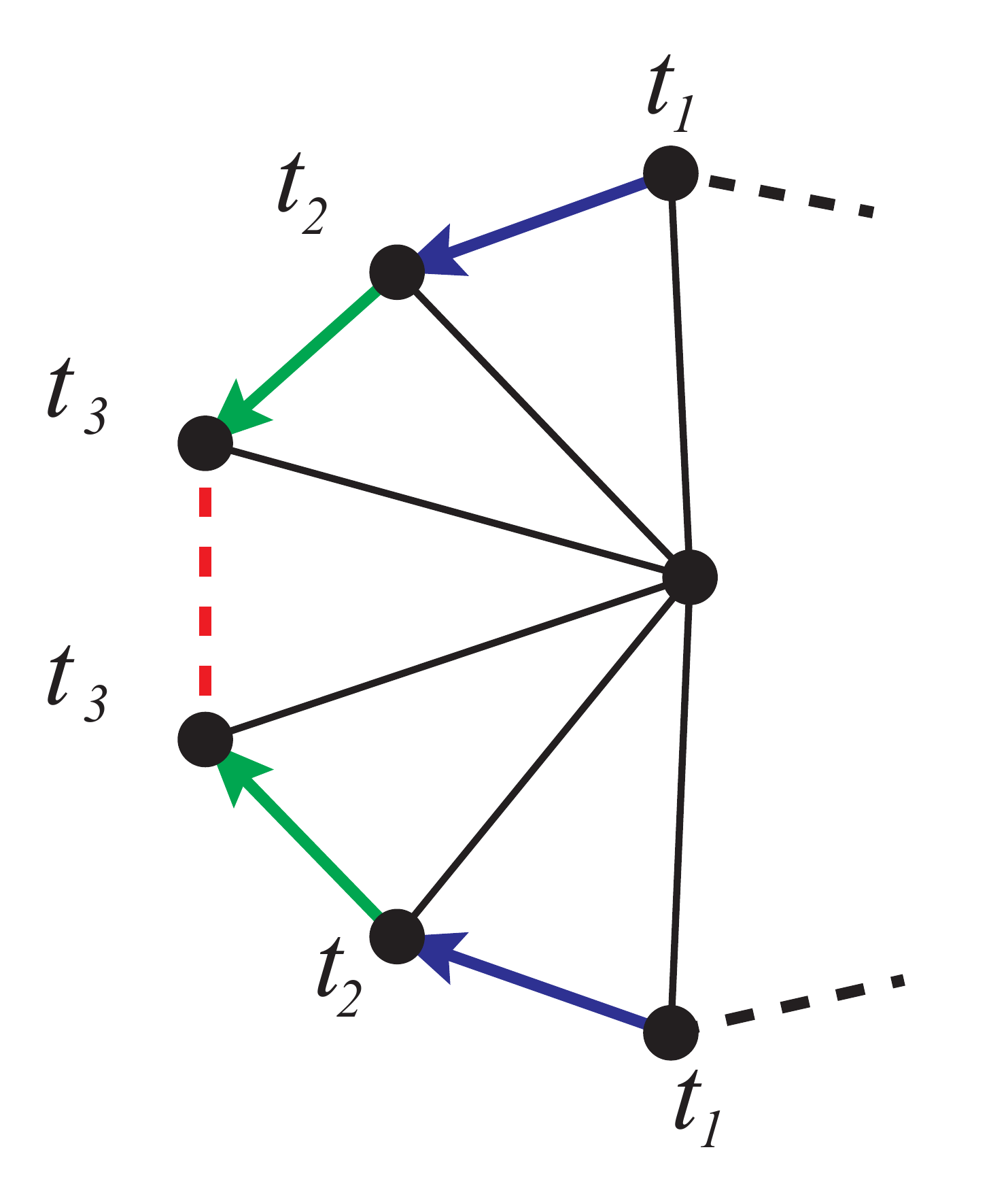}
	\caption{Two copies of $t_1$ are placed on the outer cycle. Paths along the outer cycle must be of odd or even length.}\label{Cyclepathfig}
\end{figure}

By Lemma \ref{lemmaj}, a single bond type, say $a$, must be used on both spoke edges. By the proof of Lemma \ref{cyclelemma}, each of the remaining two bond-edge types of $t_1$ must be oriented opposingly along the outer cycle (one bond-edge type, say $b$, is colored in blue in Figure \ref{Cyclepathfig}). Furthermore, since each bond-edge type has appeared twice, then by Lemma \ref{cyclelemma}, the two vertices adjacent to those of $t_1$ via the bond-edge type $b$ must be of the same tile type, say $t_2$. By the nature of the wheel graph, there must exist two paths between the two copies of $t_1$ whose union is exactly the outer cycle. 

If at least one such path has odd length, then we see on the right side of Figure \ref{Cyclepathfig} that applying Lemmas \ref{lemmaj} and  \ref{cyclelemma} to tile type $t_2$ and each subsequent tile type, eventually some tile type will be adjacent to itself, which violates Lemma \ref{lemma3}. This contradicts the assumption that the outer cycle of $W_n$ could be constructed using $\floor{\frac{n}{2}}$ tile types.

If both paths along the outer cycle are of even length, then $n$ must be odd. We see on the left side of Figure \ref{Cyclepathfig} that repeating the reasoning above, some tile type will be used exactly once, which contradicts the assumption that every tile could be used twice.

Alternatively, if the outer cycle of $W_n$ could be constructed in strictly less than $\floor{\frac{n}{2}}$ tile types, then at least one tile type would appear three or more times on a vertex on the outer cycle. Suppose without loss of generality that the tile type is $t=\{a, b, c\}$ (where $a, b, c$ may not necessarily be distinct bond-edge types). By Lemma \ref{lemma3}, tile type $t$ may not be adjacent to itself, so we know each instance of $t$ must be at least distance 2 from itself. To satisfy Lemma \ref{lemmaj}, some bond-edge type, say $a$, must be used on all the spoke edges. This leaves three edges on the outer cycle to be labeled with the bond-edge types $b$ and $c$, which violates Lemma \ref{cyclelemma}. Thus, we may not use fewer than $\floor{\frac{n}{2}}$ tile types to construct $W_n$.

Since each case shows that a pot using fewer than $\floor{\frac{n}{2}}+2$ tiles violates the conditions of Scenario 3, then $T_3(W_n) \geq \floor{\frac{n}{2}}+2$. Since the pots given in Equations (\ref{Scen3even}) and (\ref{Scen3odd}) achieve this lower bound, then $T_3(W_n) = \floor{\frac{n}{2}}+2$.
\end{proof}

\section{Conclusion}

In this paper we have presented the minimum number of distinct tile types and bond-edge types needed to construct a self-assembled DNA complex with a wheel graph structure. There are still several other infinite classes of graphs beyond wheel graphs to study, and we anticipate that the results presented here may aid future authors in finding similar results for graphs which contain $W_n$ as a subgraph.  

The minimum number of distinct tile and bond edge-types needed to construct the wheel graph in Scenarios 1 and 2 are not only equal, but also constant, i.e., they do not increase as a function of the number of vertices of the target graph. The  minima in Scenario 3 do increase as a function of $n$, which is to be expected, given the more stringent conditions that come with Scenario 3.  

 In Scenario 3, we established two lemmas that aided us in finding lower bounds for the minimum number of tile types and bond-edge types. We believe this method of finding restrictions on the construction of a given complex may be generalized or adapted to find more results for additional classes of graphs. 

\section{Acknowledgements}

This work was supported in part by the CSU San Bernardino Office of Student Research.

%%%%%%%%%%%%%%%%%%%%%%%%%%%%%%%%%

\newpage
\section*{References}
\bibliographystyle{plain} 
\bibliography{Bibliography}

\begin{thebibliography}{1}

\bibitem{2}
Leonard~M Adleman.
\newblock Molecular computation of solutions to combinatorial problems.
\newblock {\em Science}, 266(5187):1021--1024, 1994.

\bibitem{ellis2014minimal}
Joanna Ellis-Monaghan, Greta Pangborn, Laura Beaudin, David Miller, Nick Bruno,
  and Akie Hashimoto.
\newblock Minimal tile and bond-edge types for self-assembling dna graphs.
\newblock In {\em Discrete and Topological Models in Molecular Biology}, pages
  241--270. Springer, 2014.

\bibitem{43}
Hongzhou Gu, Jie Chao, Shou-Jun Xiao, and Nadrian~C Seeman.
\newblock A proximity-based programmable dna nanoscale assembly line.
\newblock {\em Nature}, 465(7295):202, 2010.

\bibitem{labean2007constructing}
Thom~H LaBean and Hanying Li.
\newblock Constructing novel materials with dna.
\newblock {\em Nano Today}, 2(2):26--35, 2007.

\bibitem{han}
Andre~V. Pinheiro, Dongran Han, William~M. Shih, and Hao Yan.
\newblock Challenges and opportunities for structural dna nanotechnology.
\newblock {\em Nature Nanotechnology}, 6(12):763--773, 2011.

\bibitem{80}
Nadrian~C Seeman.
\newblock An overview of structural dna nanotechnology.
\newblock {\em Molecular biotechnology}, 37(3):246, 2007.

\end{thebibliography}

\end{document}